\documentclass{article}
\usepackage[utf8]{inputenc}
\usepackage{graphicx}
\usepackage{amsthm}
\usepackage{amssymb}
\usepackage{amsmath}
\usepackage{setspace}
\usepackage{amsfonts}
\usepackage{comment}
\usepackage{appendix}
\usepackage{algorithm}
\usepackage[noend]{algorithmic}
\usepackage{tikz}
\usetikzlibrary{intersections, calc, arrows, positioning, arrows.meta}

\usepackage{color}

\theoremstyle{plain}
\newtheorem{theorem}{Theorem}
\newtheorem{prop}[theorem]{Proposition}
\newtheorem{lemma}[theorem]{Lemma}
\newtheorem{cor}[theorem]{Corollary}

\begin{document}

\title{Finding popular branchings in vertex-weighted digraphs}
\author{Kei Natsui\thanks{System and Information Engineering, Graduate School of Science and Technology, University of Tsukuba. \texttt{s2120448@s.tsukuba.ac.jp}} 
\and  Kenjiro Takazawa\thanks{Department of Industrial and Systems Engineering, 
    Faculty of Science and Engineering, Hosei University. 
    Supported by 
     JSPS KAKENHI Grant Number JP20K11699, Japan.
     \texttt{takazawa@hosei.ac.jp}. 
    }} 
\date{February 2022}

\maketitle
\begin{abstract}
Popular matchings have been intensively studied recently as a relaxed concept of stable matchings.
By 
applying the concept of popular matchings to branchings in directed graphs, Kavitha et al.\ introduced popular branchings. 
In a directed graph $G=(V_G,E_G)$, each vertex has preferences over its incoming edges. 
For branchings $B_1$ and $B_2$ in $G$, 
a vertex $v\in V_G$ prefers $B_1$ to $B_2$ if $v$ prefers its incoming edge of $B_1$ to that of $B_2$, 
where having an arbitrary incoming edge is preferred to having none, 
and 
$B_1$ is more popular than $B_2$ if the number of vertices that prefer $B_1$ is greater than the number of vertices that prefer $B_2$. 
A branching $B$ is called a popular branching if there is no branching more popular than $B$.
Kavitha et al.\ proposed an algorithm for finding a popular branching when the preferences of each vertex are given by a strict partial order.
The validity of this algorithm is proved by utilizing classical theorems on the duality of weighted arborescences. 
In this paper, we generalize popular branchings to weighted popular branchings in vertex-weighted directed graphs in the same manner as weighted popular matchings 
by Mestre. 
We give an algorithm for finding a weighted popular branching, which extends the algorithm of Kavitha et al., when the preferences of 
each vertex
are given by a total preorder and the weights satisfy certain conditions. 
Our algorithm includes elaborated procedures resulting from the vertex-weights, and its validity is proved by extending the argument of the duality of weighted 
arborescences.
\end{abstract}

\section{Introduction}
\label{SECintro}
\textit{Popular matchings} provide a relaxed concept of stable matchings.
Popular matchings were introduced by G\"{a}rdenfors \cite{gardenfors75}, and have been attracting intensive attention recently since Abraham et al.\ \cite{abraham07} started studying their algorithmic aspects.
In a bipartite graph, each vertex has preferences over its adjacent vertices, and a matching $M$ is \emph{more popular} than another matching $N$ if the number of vertices that prefer the adjacent vertex in $M$ to that in $N$ is greater than the number of vertices that prefer the adjacent vertex in $N$ to that in $M$. 
A matching $M$ is called a \emph{popular matching} if no matching is more popular than $M$. 
For popular matchings, several algorithms are known.
Abraham et al.\ \cite{abraham07} were the first to give an efficient algorithm determining whether a popular matching exists and finding one if exists.

There have been various other studies on popular matching in recent years, including \cite{BP10, Cseh17, dominantM17, quasiPM19, OPM09}.
Among those,
Mestre \cite{wpm06} provided an algorithm for weighted popular matching.
In the weighted popular matching problem, weights are attached to the vertices and, instead of the number of vertices, the popularity of matchings is defined by the sum of the weights of the corresponding vertices.
The algorithm 
\cite{wpm06}
runs in polynomial time regardless of whether ties are allowed or not.

By applying the concept of popular matchings to branchings in directed graphs, 
Kavitha, Kir\'{a}ly, Matuschke, Schlotter, and Schmidt-Kraepelin \cite{kavitha20} introduced \textit{popular branchings}.
In a directed graph $G=(V_G, E_G)$, each vertex has preferences over its incoming edges.
 Let $B$ and $B'$ be branchings in $G$. We say that a vertex $v\in V$ prefers $B$ to $B'$ if $v$ prefers 
 the incoming edge in $B$
 to that in $B'$, 
 where having an arbitrary incoming edge is preferred to having none.
We say that $B$ is more popular than $B'$ if the number of vertices that prefer $B$ is greater than the number of vertices that prefer $B'$. 
A branching $B$ is called a popular branching if there is no other branching more popular than $B$.

Kavitha et al.\ \cite{kavitha20} proposed an algorithm for finding a popular branching when the preferences of the vertices are given by a 
strict
partial order. 
This algorithm determines whether a popular branching exists, and if so, outputs one.
Its validity is proved by a characterization of popular branchings which utilizes the duality of weighted arborescences.
The algorithm constructs
a directed graph $D$ 
from $G$
by adding a dummy vertex $r$ as a root and an edge $(r,v)$ for each $v\in V_G$.
Each branching $B$ in $G$ is extended to an $r$-arborescence in $D$ by adding an edge $(r,v)$ for every vertex $v \in V_G$ with no incoming edge in $B$. 
They proved that an 
$r$-arborescence $A$ in $D$ is a popular arborescence
if and only if it is a minimum cost arborescence with respect to edge weights defined in a certain manner. 

For these edge weights, they further proved that, for an integral optimal solution $y\in \mathbb{R}^{2^{V_{G}}}$ of the dual problem, the laminar structure of the support $\mathcal{F}(y) = \{X\subseteq V_G : y(X)>0\}$
has at most two layers. 
This structure leads to a one-to-one correspondence between 
the support $\mathcal{F}(y)$ and the vertex set $V_G$, 
and 
the concept of \emph{safe edges}, 
which are candidates of the edges in a popular arborescence. 
The algorithm of Kavitha et al.\ \cite{kavitha20} essentially relies on this structure.
  
In this paper, we generalize popular branchings to weighted popular branchings in the same manner as weighted 
popular
matchings \cite{wpm06}. 
Each vertex $v\in V_G$ is assigned a positive integer weight $w(v)$. 
For an $r$-arborescence 
$A$ in $D$
and a vertex 
$v \in V_G$,
let $A(v)$ denote the edge in $A$ entering $v$. 
For two $r$-arborescences $A$ and $A'$ in $D$, we define an integer $\Delta_w(A,A')$ by
\begin{align*}
\Delta_w(A,A')=\sum_{v:A(v)\succ_vA'(v)}w(v)-\sum_{v:A'(v)\succ_vA(v)}w(v), 
\end{align*}
where $e\succ_v f$ denotes that $v$ prefers $e$ to $f$.
If $\Delta_w(A,A')>0$, we say that $A$ is more popular than $A'$. 
An $r$-arborescence $A$ 
in $D$
is a \emph{popular arborescence} if no arborescence is more popular than $A$.

The main contribution of this paper is an algorithm for finding a 
popular arborescence in vertex-weighted directed graphs,
which extends the algorithm of Kavitha et al.\ \cite{kavitha20}. 
Its validity builds upon a characterization of weighted popular arborescences, 
which extends that of popular branchings \cite{kavitha20}, and our algorithm includes elaborated procedures resulting from the vertex-weights.

The following two points are specific to our algorithm. 
The first is that 
the preferences of each vertex are given by a total preorder, 
while they are given by a strict partial order in \cite{kavitha20}. 
The second is that it requires an assumption on the vertex weights: 
\begin{quote}
for any vertex $s, t, u\in V_G$, 
\begin{align*}
w(s)+w(t)>w(u)
\end{align*}
holds.
\end{quote}
Under this assumption, we can derive that the laminar structure of the support of an integer dual optimal solution $y$ has at most two layers.
By virtue of this laminar structure, 
we can define 
the one-to-one correspondence between 
$\mathcal{F}(y)$ and 
$V_G$, 
and 
safe edges 
in the same manner as \cite{kavitha20}, 
which are essential in designing the algorithm. 

Let us mention 
an application of popular branchings in the context of a voting system, 
and 
what is offered by the generalized model of weighted  popular branchings.  
Kavitha et al.\ \cite{kavitha20} suggested an application in a voting system called liquid democracy.
This is a new voting system that lies between representative democracy and direct democracy. 
In liquid democracy, voters can choose to vote themselves or to delegate their votes to the judgment of others who they believe in, and their votes flow over a network, 
constructing a fluid voting system.
In this system,
a popular branching amounts to a reasonable delegation process. 
Here, if we take vertex weights into account, it
represents 
a situation where there is a difference in voting power.
That is, 
weighted popular branchings 
are of help 
when each voter has distinct voting power, and we want to make a decision based on the total voting power rather than the number of votes.

This paper is organized as follows. 
Section \ref{SECdefwpb} formally 
defines 
weighted popular branchings.
In Section \ref{SECproperties}, in preparation for algorithm design, 
we analyze some properties of weighted popular branchings and introduce safe edges. 
In Section \ref{SECalgorithm}, we present our algorithm for finding a weighted popular branching and prove its correctness.

\section{Definition of weighted popular branchings}
\label{SECdefwpb}
Let $G=(V_G, E_G)$ be a directed graph, where every vertex has a positive integer weight $w(v)$ and preferences over its incoming edges.
The preferences of each vertex $v$ are given by a \emph{total preorder}  $\precsim_v$ on the set of edges that enter $v$.

Recall that a total preorder is defined in the following way.
Let $S$ be a finite set.
A binary relation $R$ on $S$ is \emph{transitive} if, for all $a, b, c\in S$, $aRb$ and $bRc$ imply $aRc$.
Also, $R$ is \emph{reflexive} if  $aRa$ holds for all $a\in S$. 
A relation $R$ is called a \emph{preorder} if $R$ is transitive and reflexive. 
In addition, $R$ is a \emph{total relation} if  $aRb$ or $bRa$ holds for all $a,b\in S$. 
That is, 
a total preorder is a relation which is transitive, reflexive, and total. 
Note that a partial order is a preorder, 
whereas it is not necessarily a total preorder, 
and a total preorder is not necessarily a partial order.

Let $e$ and $f$ be two edges entering the same vertex $v$.
Then, $e\precsim_v f$ means that 
$f$ has more or the same priority than $e$. 
If 
both
$e\precsim_v f$ and $f\precsim_v e$ holds, we denote it by $e\sim_v f$,
indicating that $v$ is indifferent between $e$ and $f$.
Note that 
$\sim_v$ is an equivalence relation.
Furthermore, if $e\precsim_v f$ holds 
but
$f\precsim_v e$
does not,
we denote it by $e\prec_v f$. This indicates that vertex $v$ 
strictly prefers
$f$ to $e$.
If an edge 
$f$
is strictly preferred to 
$e$, 
then we say that 
$f$ \emph{dominates} $e$.

Instead of discussing branchings in $G$, 
we mainly handle arborescences in an auxiliary directed graph $D$.
Recall that the directed graph $D$ is constructed from $G$
by adding a dummy vertex $r$ as a root and an edge $(r,v)$ for each $v\in V_G$. That is, $D$ is represented as $D=(V, E)$, where
\begin{align*}
V=V_G\cup \{r\}, E=E_G \cup \{(r,v) : v\in V_G \}.
\end{align*}
For each vertex $v\in V_G$, let $\delta^-(v)\subseteq E$ be the set of edges in $D$ that enter $v$, 
and make $(r, v)$ the least preferred incoming edge in $\delta^-(v)$. 
That is, 
every edge in $E_G \cap \delta^-(v)$ dominates $(r,v)$ for each $v \in V_G$.

An $r$-arborescence in $D$ is an out-tree with root $r$.
Recall the following notation and definition mentioned in Section 1. 
For an $r$-arborescence $A$ in $D$ and $v\in V_G$, let $A(v)$ 
denote the edge in $A$ entering $v$. 
For $r$-arborescences $A$ and $A'$ in $D$, we define $\Delta_w(A,A')$ by
\begin{align}
\Delta_w(A,A')=\sum_{v:A(v)\succ_vA'(v)}w(v)-\sum_{v:A'(v)\succ_vA(v)}w(v).
\end{align}
An $r$-arborescence $A$ is more popular than $A'$ if $\Delta_w(A,A')>0$. If no $r$-arborescence is more popular than $A$, then we say that $A$ is a popular arborescence. 
Our primary goal is to find a popular arborescence in $D$. 

\section{Properties of weighted popular branchings}
\label{SECproperties}
\subsection{Characterizing weighted popular arborescences}

In this subsection, 
by extending the argument in \cite{kavitha20}, 
we give a characterization of popular arborescences (Proposition \ref{lemma14}) by utilizing the duality theory of weighted arborescences.
We then investigate the structure of dual optimal solutions with a certain property (Proposition \ref{prop18}).
Let $A$ be an $r$-arborescence 
in $D$. 
For each edge $e=(u,v)$ in $D$, we define the cost $c_A(e)$ as follows:
\begin{align}
c_A(e)=
\begin{cases}\label{cost}
0&\quad (e\succ_vA(v)), \\
w(v)&\quad (e\thicksim_vA(v)), \\
2w(v)&\quad (e\prec_vA(v)).
\end{cases}
\end{align}
Since $c_A(e)=w(v)$ for every $e\in A$, we have $c_A(A) = w(V_G)$. 
For an arbitrary $r$-arborescence $A'$ in $D$, the following holds:
\begin{align*}
c_A(A')&=\sum_{v: A(v)\succ_v A'(v)}2w(v)+\sum_{v: A(v)\sim_v A'(v)}w(v)+\sum_{v: A(v)\prec_v A'(v)}0\\
&=
w(V_G)
+\sum_{v: A(v)\succ_v A'(v)}w(v)-\sum_{v: A(v)\prec_v A'(v)}w(v)\\
&=
c_A(A)
+\Delta_w(A,A').
\end{align*}
We thus obtain the following proposition.
An $r$-arborescence is called a \emph{min-cost $r$-arborescence} if the sum of the costs of all edges is the smallest among the $r$-arborescences.
\begin{prop}\label{prop12}
An $r$-arborescence $A$ is popular if and only if it is a min-cost $r$-arborescence in $D$ with respect to the edge costs 
$c_A$.
\end{prop}
Based on Proposition \ref{prop12}, consider the following linear program (LP1), which describes the min-cost $r$-arborescence problem, and its dual (LP2).
For any non-empty set $X\subseteq V_G$, let $\delta^-(X) \subseteq E$ be the set of edges in $D$ that enter $X$. 
\begin{align}
&\text{(LP1)}  &&\text{minimize} &&\sum_{e\in E}c_A(e)\cdot x(e) \\
&&&\text{subject to} &&\sum_{e\in \delta^{-}(X)}x(e) \geq 1 && (X\subseteq V_G, X\neq\emptyset),\\
&&&&&x(e)\geq0 && (e\in E).
\end{align}
\begin{align}
&\text{(LP2)}  &&\text{maximize} &&\sum_{X\subseteq V_G, X\neq\emptyset}y(X) \\
&&&\text{subject to} &&\sum_{X: e\in \delta^{-}(X)}y(X) \leq c_A(e) && (e\in E), \label{lp2st1new}\\ 
&&&&&y(X) \geq 0 && (X\subseteq V_G,X\neq\emptyset).
\end{align}
For any feasible solution $y$ to (LP2), let $\mathcal{F}(y)=\{X\subseteq V_G : y(X)>0\}$ be the support of $y$. 

The following proposition is a direct consequence of the definition \eqref{cost} of the costs $c_A$.
\begin{prop}
\label{prop19}
For an $r$-arborescence $A$ and a 
feasible solution $y$ to (LP2), 
we have that $$\sum_{X\colon v \in X}y(X)\leq 2w(v)
\qquad (v \in V_G). $$
\end{prop}
\begin{proof}
For every vertex $v\in V_G$, it holds that 
\begin{align}
\sum_{X \colon v\in X} y(X) = \sum_{X : (r,v)\in \delta^-(X)}y(X) \leq c_A(r,v)\leq 2w(v).
\end{align}
\end{proof}
Furthermore, the following proposition can be derived by the duality of weighted arborescences, as described in \cite{kavitha20}.
\begin{prop}
\label{lemma14}
For an $r$-arborescence $A$, and an optimal solution $y_A^*$ to (LP2) the following statements are equivalent.
\begin{enumerate}
\item[(i)] $A$ is a popular arborescence.
\item[(ii)] $\sum_{X\subseteq V_G}y^*_A(X)=w(V_G)$.
\item[(iii)] $|A \cap \delta^-(X)|=1$ for all $X\in \mathcal{F}(y_A^*)$ and $\sum_{X:e\in \delta^-(X)}y^*_A(X)=w(v)$ for all $e=(u,v) \in A$.
\end{enumerate} 
\end{prop}

From now on, 
we deal with an optimal solution $y$ with certain properties. 
The first property is described in the following lemma.
A set family $\mathcal{F}$ is called \emph{laminar} if for any two sets $X, Y\in \mathcal{F}$, at least one of the three sets $X\setminus Y, Y\setminus X, X\cap Y$ is empty.

\begin{lemma}[\cite{EandG77, frank79b,fulkerson74}]
\label{lemma13}
If 
the costs $c_A$ are integers,
there exists an integral optimal solution $y^*$ to (LP2) such that $\mathcal{F}(y^*)$ is laminar.
\end{lemma}
Since
the costs $c_A(e)\in \{0, w(v), 2w(v)\}$ 
are integers, it follows from 
Lemma \ref{lemma13} that there exists an integral optimal solution  $y^*_A$ to (LP2) such that $\mathcal{F}(y^*_A)$ is laminar. 

Let $A \subseteq E$ be an $r$-arborescence and $y_A^*$ be an optimal solution for (LP2)\@.
Here, we consider the properties of $y_A^*$ and $\mathcal{F}(y^*_A)$.
Let $E^{\circ}$ be the set of edges $e\in E$ satisfying $\sum_{X: e\in \delta^-(X)}y^*_A(X)=c_A(e)$
and let $D^{\circ}=(V, E^{\circ})$.
For a directed graph $D'=(V', E')$ and 
its vertex subset $X \subseteq V'$, 
the subgraph induced by $X$ is denoted by 
$D'[X] = (X, E'[X])$. 
Similarly, 
for an edge subset $A' \subseteq E'$, 
the set of edges in $A'$ induced by $X$ is denoted by 
$A'[X]$. 
The following lemma applies to general weighted arborescences. 

\begin{lemma}
\label{lemma16}
For an $r$-arborescence $A \subseteq E$, 
there exists 
an integral optimal solution
$y_A^*$ 
to (LP2)
such that 
$\mathcal{F}(y_A^*)$ is laminar and
$D^{\circ}[Y]$ is strongly connected for every $Y\in \mathcal{F}(y_{A}^*)$.
\end{lemma}
\begin{proof}
Let $y_A^*$ be an optimal integral solution to (LP2) whose support is laminar. Among those, choose $y_A^*$ such that $\sum_{Y\in \mathcal{F}(y_A^*)}|Y|\cdot y_A^*(Y)$ is minimal.
Assume to the contrary that $D^\circ[Y]$ is not strongly connected for some $Y\in \mathcal{F}(y^*_A)$. 
Then there exists a strongly connected component $Z$ of $D^\circ[Y]$ such that $\delta^-(Z)\cap E^{\circ}[Y]=\emptyset$, that is, 
for all edges $e$ in $\delta^-(Z)$, $\sum_{X:e\in \delta^-(X)}y(X)<c_A(e)$ holds.
Define $y'\in \mathbb{R}^{{2^{V_G}}\setminus \{\emptyset\}}$ by 
\begin{align}
y'(X)=
\begin{cases}
y^*_{A}(X) + 1 &(X=Z), \\
y^*_{A}(X) - 1 &(X=Y),\\
y^*_{A}(X) &(\mathrm{otherwise}).
\end{cases}
\end{align}
Then, 
$\sum_{X:e\in \delta^-(X)}y'(X) \leq c_A(e)$ for all edges in $\delta^-(Z)$. Thus $y'$ is a feasible solution to (LP2) and since the objective function values for $y^*_{A}$ and $y'$ are equal, $y'$ is also an optimal solution to (LP2).
Since 
$\sum_{Y\in \mathcal{F}(y')}|Y| \cdot y'(Y) < \sum_{Y\in \mathcal{F}(y_A^*)}|Y|\cdot y_A^*(Y)$, 
this contradicts the minimality of $\sum_{Y\in \mathcal{F}(y^*_A)}|Y|\cdot y_A^*(Y)$. Therefore we conclude that $D^\circ[Y]$ is strongly connected for each $Y\in \mathcal{F}(y_A^*)$.
\end{proof}

In what follows, 
we denote by $y_A^*$ the integer optimal solution to (LP2) described in Lemma \ref{lemma16}.
In addition to Lemma \ref{lemma16}, if $A$ is a popular arborescence, 
we can impose a stronger condition on $y_A^*$. 

\begin{lemma}\label{lemma17}
For a popular arborescence $A$, there exist 
an integral optimal solution $y_A^*$ to (LP2) 
such that 
$\mathcal{F}(y_A^*)$ is laminar,  $D^{\circ}[Y]$ is strongly connected for every $Y\in \mathcal{F}(y_{A}^*)$, 
and 
the following is satisfied.

\begin{quote}
For $Y\in \mathcal{F}(y_A^*)$, let $Y_1, \ldots, Y_k$ be the sets in $\mathcal{F}(y_A^*)$ that are maximal proper 
subsets
of $Y$. Then,  
\begin{align}
\label{prop17eq}
|Y\setminus(Y_1\cup\cdots\cup Y_k)|=1.
\end{align}
\end{quote}
\end{lemma}
\begin{proof}
 Among 
 $y_A^*$ 
 described in
 Lemma \ref{lemma16}, 
 consider $y_A^*$ for which $|\mathcal{F}(y_A^*)|$ is minimal.
 Let $Y\in \mathcal{F}(y^*_{A})$ and $Y'=Y\setminus (Y_1\cup\cdots\cup Y_k)$. 
We first show that $|Y'|\leq 1$. Since $c_A(e)=w(v)>0$ for $e=(u,v)\in A$, it follows from Proposition \ref{lemma14}(iii) that $|A[Y']|=0$ and $|A\cap \delta^-(Y')\cap\left(\bigcup_{i=1, \ldots, k}\delta^+(Y_i)\right)|=0$.
It is also derived from Lemma \ref{lemma14}(iii) that $|A\cap\delta^-(Y)|=1$, and thus
\begin{align}
\notag
|Y'|&=|A\cap\delta^-(Y)\cap\delta^-(Y')|+|A[Y']|+\left| A\cap \delta^-(Y')\cap\left(\bigcup_{i=1,\ldots, k}\delta^+(Y_i)\right)\right|\\
\notag
&=|A\cap\delta^-(Y)\cap\delta^-(Y')|\\
\notag
&\leq|A\cap\delta^-(Y)\cap\delta^-(Y')|+\left| A\cap\delta^-(Y)\cap\left( \bigcup_{i=1,\ldots k}\delta^-(Y_i)\right)\right|\\
\label{|Y'|=1}
&=|A\cap \delta^-(Y)|
=1.
\end{align}

Suppose to the contrary that $Y'=\emptyset$. Let $e=(u,v)\in A\cap \delta^-(Y)$.
First, we consider the case where 
$\{v\}\in \mathcal{F}(y_A^*)$. 
It follows from $e\in A$ that $c_A(e)=w(v)$. 
For any edge $(u',v)\in \delta^-(v)$, it holds that $c_A(u',v)\geq y^*_{A}(\{v\})>0$ from the constraint (\ref{lp2st1new}).
 We know that $c_A(u',v)\in \{2w(v), w(v), 0\}$, and hence $c_A(u',v)\geq w(v)$ holds. 
Let $y^*_{A}(Y)=\alpha$ and $y^*_{A}(\{v\})=\beta$. Then $w(v)=c_A(e)\geq y^*_{A}(Y)+y^*_{A}(\{v\})=\alpha+\beta$. 
Thus, if we replace these two values with $y^*_{A}(Y)=0$ and $y^*_{A}(\{v\})=\alpha+\beta$, we can construct $y^*_{A}$ with a smaller value of $|\mathcal{F}(y_A^*)|$ while satisfying the constraints of (LP2). 
This contradicts the minimality of $|\mathcal{F}(y_A^*)|$.

Next, consider the case where $\{v\}$ is not included in $\mathcal{F}(y_A^*)$.  
In this case, $|Y|\geq 2$ holds. 
It then follows from Lemma \ref{lemma15} that $Y$ is not minimal in $\mathcal{F}(y_A^*)$, that is, $Y_1, \ldots, Y_k \in \mathcal{F}(y^*_{A})$ always exist. 
Since 
$Y'=\emptyset$, without loss of generality we can assume that $v\in Y_1 \in \mathcal{F}(y_A^*)$. Note that $|Y_1|\geq 2$ since $\{v\} \notin \mathcal{F}(y_A^*)$. Now, since we have chosen $y^*_{A}$ described in
Lemma \ref{lemma16}, $D^\circ[Y]$ is strongly connected. 
Therefore, there exists $e'=(u',v')\in E^\circ[Y]\cap \delta^-(Y_1)$. 
Similarly, $D^\circ[Y_1]$ is also strongly connected, and hence there exists $e''=(u'',v')\in E^\circ[Y_1]\cap \delta^-(v')$. 
Then, we have $\{X\in \mathcal{F}(y_A^*) : e''\in X\}\subsetneq \{X\in \mathcal{F}(y_A^*) : e'\in X\}$. 
If $v'\neq v$, since $r\notin Y$ and 
each edge $e \in E^{\circ}$ satisfies
$\sum_{X: e\in \delta^-(X)}y(X)=c_A(e)$, it holds that $c_A(e'')<c_A(e')<c_A(r,v')$, which implies  $c_A(e'')=0$ and $c_A(e')=w(v')$.
Now, since $e''\in E^\circ$, there is no edge $e^*\in E^\circ[Y_1]\cap\delta^-(v')$ such that $c_A(e^*)=w(v')$. 
Therefore $A(v')$ cannot be taken from $E^\circ [Y_1]$, contradicting that $|Y_1|\geq 2$ and $|A\cap \delta^-(Y_1)|=1$. 
If $v'=v$, it follows from $e=(u,v)\in \delta^-(Y)$ that $c_A(e'')<c_A(e')<c_A(e)=w(v)$ as in the previous discussion. However, this contradicts the 
definition 
of cost $c_A$: $c_A(e'), c_A(e'')\in \{0, w(v), 2w(v)\}$.

From the above, there is no such $Y\in \mathcal{F}(y_A^*)$ for which $Y'=\emptyset$, and from (\ref{prop17eq}), we conclude that $|Y'|=1$.
\end{proof}

Furthermore, we can derive the following lemma holds.
\begin{lemma}
\label{lemma15}
If $A$ is a popular arborescence and a set $Y\in \mathcal{F}(y_A^*)$ is minimal in $\mathcal{F}(y_A^*)$, then $|Y|=1$.
\end{lemma}
\begin{proof}
Suppose to the contrary that $|Y|\geq 2$ for a minimal set $Y$ in $\mathcal{F}(y_A^*)$. 
Since $Y\in \mathcal{F}(y_A^*)$, by Proposition \ref{lemma14}(iii), we have $|A\cap \delta^-(Y)|=1$. 
Also, since $A$ is an $r$-arborescence, it holds that $|A\cap(\bigcup_{v\in Y}\delta^{-}(v))|=|Y|$, and hence $|A[Y]|=|A\cap(\bigcup_{v\in Y}\delta^{-}(v))|-|A\cap \delta^-(Y)|=|Y|-1\geq 1$.
 It then follows that  $A[Y] \neq \emptyset$, and let $e$ be an edge in $A[Y]$. Now $Y$ is a minimal set in $\mathcal{F}(y_A^*)$, implying that $e\notin \delta^-(Y')$ for any $Y'\in \mathcal{F}(y_A^*)$. This contradicts Proposition \ref{lemma14}(iii).
\end{proof}

From Lemmas \ref{lemma17} and \ref{lemma15}, the next proposition follows.

\begin{prop}
\label{prop18}
Let $A$ be a popular arborescence. Then, there exists a one-to-one correspondence between the sets in $\mathcal{F}(y_A^*)$ satisfying (\ref{prop17eq}) and the vertices in $V_G$. 
Moreover, for each $X\in \mathcal{F}(y^*_{A})$ and the terminal vertex $v$ of the edge $(u,v)\in A\cap \delta^-(X)$, we have $y^*_{A}(X)=w(v)$. 
\end{prop}
Denote by $Y_v$ the unique set in $\mathcal{F}(y^*_A)$ that is in correspondence with $v$ in the sense of Proposition \ref{prop18}.
Note that  
$\mathcal{F}(y_A^*)=\{Y_v:v\in V_G\}$ 
and
the unique edge in $A$ entering $Y_v$ is $A(v)$.
We thus refer to $v$ as the
\emph{entry-point}
of $Y_v$.
An example of these correspondences is shown in Figure \ref{entrypoint}.

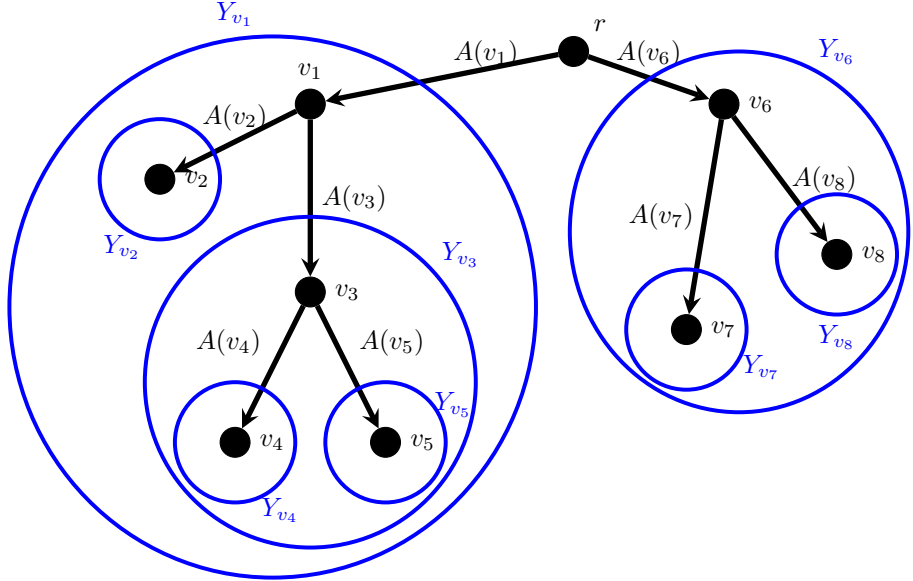
\begin{figure}
\begin{center}
\begin{tikzpicture}

\node (r) [fill, draw, circle, inner sep=4, label=above right:$r$] at (7.5,7.7){};
\node (v1) [fill, draw, circle, inner sep=4, label=above:$v_1$] at (4,7){};
\node (v2) [fill, draw, circle, inner sep=4, label=right:$v_2$] at (2,6){};
\node (v3) [fill, draw, circle, inner sep=4, label=right:$v_3$] at (4,4.5){};
\node (v4) [fill, draw, circle, inner sep=4, label=right:$v_4$] at (3,2.5){};
\node (v5) [fill, draw, circle, inner sep=4, label=right:$v_5$] at (5,2.5){};
\node (v6) [fill, draw, circle, inner sep=4, label=right:$v_6$] at (9.5,7){};
\node (v7) [fill, draw, circle, inner sep=4, label=right:$v_7$] at (9,4){};
\node (v8) [fill, draw, circle, inner sep=4, label=right:$v_8$] at (11,5){};

\path[line width=2, ->,>=stealth] 
(r) edge[above right] node{$A(v_1)$} (v1)
(r) edge[above] node{$A(v_6)$} (v6)
(v1) edge[above] node{$A(v_2)$} (v2)
(v1) edge[right] node{$A(v_3)$} (v3)
(v3) edge[above left] node{$A(v_4)$} (v4)
(v3) edge[above right] node{$A(v_5)$} (v5)
(v6) edge[left] node{$A(v_7)$} (v7)
(v6) edge[right] node{$A(v_8)$} (v8);

\draw [line width=1.5, blue] (v2) circle [x radius=0.8, y radius=0.8];
\draw [line width=1.5, blue] (v4) circle [x radius=0.8, y radius=0.8];
\draw [line width=1.5, blue] (v5) circle [x radius=0.8, y radius=0.8];
\draw [line width=1.5, blue] (v7) circle [x radius=0.8, y radius=0.8];
\draw [line width=1.5, blue] (v8) circle [x radius=0.8, y radius=0.8];
\draw [line width=1.5, blue] (3.5,4.3) circle [x radius=3.5, y radius=3.6];
\draw [line width=1.5, blue] (4,3.3) circle [x radius=2.2, y radius=2.2];
\draw [line width=1.5, blue] (9.7,5.3) circle [x radius=2.25, y radius=2.4];

\node (Yv) [text=blue] at (3,8.2) {$Y_{v_1}$};
\node (Yv) [text=blue] at (1.5,5.1) {$Y_{v_2}$};
\node (Yv) [text=blue] at (6,5) {$Y_{v_3}$};
\node (Yv) [text=blue] at (3.6,1.6) {$Y_{v_4}$};
\node (Yv) [text=blue] at (5.9,3) {$Y_{v_5}$};
\node (Yv) [text=blue] at (11,7.7) {$Y_{v_6}$};
\node (Yv) [text=blue] at (10,3.5) {$Y_{v_7}$};
\node (Yv) [text=blue] at (11,3.9) {$Y_{v_8}$};
\end{tikzpicture}
\end{center}
\caption{One-to-one correspondence between $v\in V_G$ and $Y_v\in \mathcal{F}(y_A^*)$}
\label{entrypoint}
\end{figure}

\subsection{Weight assumption and safe edges}

In Kavitha et al.'s algorithm for finding popular branching \cite{kavitha20},
the laminar structure of $\mathcal{F}(y^*_{A})$ has at most two layers: 
\begin{align}
\label{EQtwolayer}
|\{X\in \mathcal{F}(y^*_{A}) \mid v\in X\}| \leq 2 \quad (v \in V_G).
\end{align}
This structure plays an important role in the algorithm.
By \eqref{EQtwolayer},
we can construct the vertex set corresponding to $X \in \mathcal{F}(y_{A}^*)$ in the algorithm using the concept of safe edge.
This can be utilized to prove the validity of the algorithm. 

In the unweighted case, 
\eqref{EQtwolayer} follows from 
Proposition \ref{prop19}. 
However, when the weights $w(v)$ can be more than one, Proposition \ref{prop19} alone does not rule out the case where $|\{X \in \mathcal{F}(y_A^*) \mid v\in X\}| \geq 3$. 
In order to maintain \eqref{EQtwolayer},
as mentioned in Section \ref{SECintro}, 
we impose 
an assumption 
on the vertex weights. 
Recall that the assumption is:
\begin{align}
\label{3.11}
w(s)+w(t)>w(u) \quad (s, t, u\in V_G).
\end{align}

From this assumption, 
we can derive the following proposition.

\begin{prop}
\label{prop20}
Let $A$ be a popular arborescence 
and let $y_A^*$ satisfy (\ref{prop17eq}).
If the condition (\ref{3.11}) holds for any three vertices $s, t, u \in V_G$, 
then $|\{X\in \mathcal{F}(y_A^*) \mid v\in X\}|\leq 2$ holds for every vertex $v\in V_G$.
\end{prop}

\begin{proof}
Assume to the contrary that $|X\in \mathcal{F}(y^*_{A})\mid v\in X\}| \geq 3$ for some $v\in V_G$.
In this case, there exist two vertices $a,b\in V_G \setminus \{v\}$ such that $v \in Y_a\cap Y_b \cap Y_v$.
Since $y^*_{A}(Y_a)=w(a), y^*_{A}(Y_b)=w(b), y^*_{A}(Y_v)=w(v)$ from Proposition \ref{prop18} and $w(v) < w(a)+w(b)$ from assumption (\ref{3.11}), we have $c_A(r,v)\leq 2w(v)<y^*_{A}(Y_a)+y^*_{A}(Y_b)+y^*_{A}(Y_v)$, which contradicts the constraint (\ref{lp2st1new}) in (LP2). Therefore, under assumption (\ref{3.11}), $|X\in \mathcal{F}(y^*_{A})\mid v\in X\}| \leq 2$ holds for every $v\in V_G$. 
\end{proof}
From Propositions \ref{prop18} and \ref{prop20}, and the laminarity of $\mathcal{F}(y_A^*)$, the following corollary can be derived.
An example of this corollary is shown in Figure \ref{one-to-one}.

\begin{cor}
\label{cor21}
Under assumption (\ref{3.11}), for $y^*_{A}$ satisfying (\ref{prop17eq}) 
and $v\in V_G$ with $|Y_v|\geq 2$, 
it holds that $Y_u=\{u\}$ for each 
$u\in Y_v\setminus \{v\}$.
\end{cor}

\begin{figure}
\begin{center}
\begin{tikzpicture}

\node (another) [fill, draw, circle, inner sep=4] at (9,8){};
\node (v) [fill, draw, circle, inner sep=4, label=right:$v$] at (6,6){};
\node (s) [fill, draw, circle, inner sep=4, label=right:$s$] at (4,4){};
\node (t) [fill, draw, circle, inner sep=4, label=right:$t$] at (6,3){};
\node (u) [fill, draw, circle, inner sep=4, label=right:$u$] at (8,4){};

\path[line width=2, ->,>=stealth] 
(another) edge[right] node{$A(v)$} (v)
(v) edge[left] node{$A(s)$} (s)
(v) edge[left] node{$A(t)$} (t)
(v) edge[left] node{$A(u)$} (u);

\draw [line width=1.5, blue] (s) circle [x radius=1, y radius=0.8];
\draw [line width=1.5, blue] (t) circle [x radius=1, y radius=0.8];
\draw [line width=1.5, blue] (u) circle [x radius=1, y radius=0.8];
\draw [line width=1.5, blue] (6,4) circle [x radius=3.5, y radius=2.5];

\node (Yv) [text=blue] at (5.8,6.8) {$Y_v$};
\node (Ys) [text=blue] at (4,2.9) {$Y_s$};
\node (Yt) [text=blue] at (7,2.2) {$Y_t$};
\node (Yu) [text=blue] at (8.8,3) {$Y_u$};
\end{tikzpicture}
\end{center}
\caption{An example of Corollary \ref{cor21}:
For each 
$u\in Y_v\setminus \{v\}$, it holds that $Y_u=\{u\}$.}
\label{one-to-one}
\end{figure}
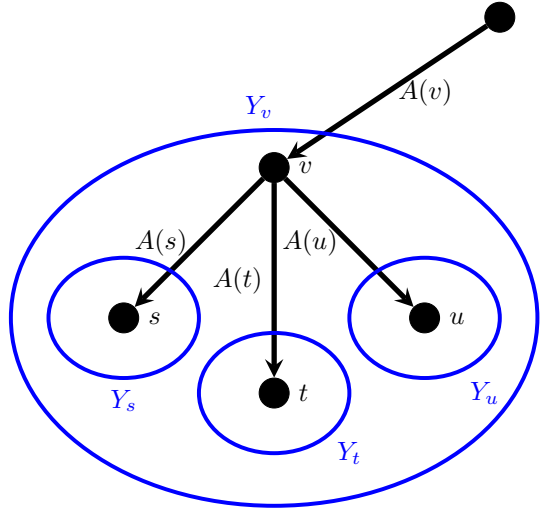

The safe edges used in our algorithm are defined in the same way as \cite{kavitha20}, described below.
For $X\subseteq V_G$, an edge $(u,v)\in E[X]$ satisfying the following two conditions is called a \emph{safe edge} in $X$, and the set of safe edges in $X$ is denoted by $S(X)$:
\begin{enumerate}
\item $(u,v)$ is not dominated by any edges $E[X]$, i.e., $(u,v)\succsim_v (u',v)$ for all $(u',v) \in E[X]$; and
\item $(u,v)$ dominates each $(t,v)$ with $t \notin X$, i.e., $(u,v)\succ_v (t,v)$ for all $(t,v)\in \delta^-(X)$.
\end{enumerate}
Recall that the preferences of each vertex are given by the total preorder.
Hence, if there exists $e \in \delta^-(v)\cap \delta^-(X)$ such that $e$ is one of the most preferred edges in $\delta^-(v)$, it holds that $S(X)\cap \delta^-(v)=\emptyset$.
Otherwise, $S(X)\cap \delta^-(v)$ is the set of the most preferred edges in $\delta^-(v)$.
 
The edges in a popular arborescence 
are basically chosen from safe edges, 
as shown in the next proposition. 

\begin{prop}
For any popular arborescence $A$ and $X\in\mathcal{F}(y_A^*)$ satisfying (\ref{prop17eq}), it holds that $A\cap E[X] \subseteq S(X)$.
\end{prop}
\begin{proof}
Assume to the contrary that there exists an edge $(u,v)\in (A\cap E[X])\setminus S(X)$. 
It follows from $(u,v)\in A$ that $c_A(u,v)=w(v)$. 
Also, $(u,v) \notin S(X)$ 
implies that there exists an edge 
$(u',v) \in E[X]$
that dominates $(u,v)$ or an edge 
$(t,v)\in \delta^-(X)$
that is not dominated by $(u,v)$.

Suppose that there exists an edge $(u',v)\in E[X]$ 
dominating
$(u,v)$. 
In this case, $c_A(u',v)=0$.
Here $v\in X$ is not an entry-point of $X$, and thus $Y_v=\{v\} \in \mathcal{F}(y_A^*)$ follows from Corollary \ref{cor21}, which does not satisfy the constraint (\ref{lp2st1new})
for $(u',v)$.

Suppose that there exists an edge $(t,v)\in \delta^-(X)$ not dominated by $(u,v)$. 
Since $(u,v) \precsim_v (t,v)$, we have $c_A(t,v)\in \{0, w(v)\}$. 
Also, since $|A\cap \delta^-(X)|=1$ and $(u,v)\in A\cap E[X]$, there exist an edge in $A\cap \delta^-(X)$ whose terminal vertex $s\in X \setminus \{v\}$ is the entry-point for $X$. 
Then, by Proposition \ref{prop18}, $y^*_{A}(X)=w(s)$ holds. 
Now, $(t,v)$ enters the two sets 
$\{v\}$ and $X$ in $\mathcal{F}(y_A^*)$.
It thus follows that $y^*_{A}(X)+y^*_{A}(\{v\}) = w(s)+w(v) >w(v) \geq c_A(t,v)$, which contradicts the constraint (\ref{lp2st1new}) 
for $(t,v)$.
Therefore, we conclude that $A\cap E[X] \subseteq S(X)$ for any popular arborescence $A$ and $X\in\mathcal{F}(y_A^*)$.
\end{proof}

\section{Weighted popular branching algorithm}
\label{SECalgorithm}
We are now ready to describe our algorithm for finding a weighted popular arborescence 
and prove its validity. 
The algorithm is described in Algorithm 1.

\begin{algorithm}
\caption{Finding popular branching algorithm}
\begin{algorithmic}[1]
\FOR{each $v\in V_G$}
\STATE $X_v^0=V_G, i=0$
\WHILE{$v$ does not reach all vertices in the graph $D_v^i=(X_v^i,S(X_v^i))$}
\STATE $X_v^{i+1}=$ the set of vertices reachable from $v$ in $D_v^i$
\STATE $i=i+1$
\ENDWHILE
\STATE $X_v=X_v^i$, $D_v=D_v^i$
\ENDFOR
\STATE  $\mathcal{X}=\{X_v: v\in V_G\}$,
$\mathcal{X}'=\{X_v\in \mathcal{X}: \mbox{$X_v$ is maximal in $\mathcal{X}$}\}$, $E'=\emptyset$, and  
$D'=(\mathcal{X}'\cup \{r\}, E')$
\FOR{each $X_v\in \mathcal{X}'$}
\STATE let 
$\bar{X}_v$ be the strongly connected component of $D_v$ such that no edge in $S(X_v)$ enters.
    \IF{every $v' \in \bar{X}_v$ which has minimum weight in $\bar{X}_v$ satisfies the following condition
    
    \begin{quote}
There exist a vertex $s\in X_v\setminus \bar{X}_v$ and an edge $f\in(E[X_v]\setminus S(X_v))\cap\delta^-(v')$ such that\\
\vspace{-3pt}
 \begin{itemize}
     \item[1.] $w(s)<w(v')$,\\
     \vspace{-3pt}
     \item[2.] $v'$ is reachable from $s$ by the edges in $S(X_v)\cup \{f\}$,\\
     \vspace{-3pt}
     \item[3.] $f \succ_{v'} e$ holds for all $e\in \delta^-(v')\cap \delta^-(X_v)$.
 \end{itemize}
 \vspace{-3pt}
\end{quote}
}
    \RETURN{``No popular arborescence in $D$."}

    \ELSE 
        \FOR{every $v' \in \bar{X}_v$ which has minimum weight in $\bar{X}_v$}
            \IF{$e=(u,v')$ is not dominated by an edge in $\delta^-(X_v)$ for every $u\notin X_v$}
            \STATE define an edge $e'$ in $D'$
                by
                \begin{align}
                e'=
                \begin{cases}
                (U,X_v)  &(u\in U, U\in\mathcal{X}),\\
                (r,X_v) &(u=r),
                \end{cases}
                \end{align}
            \STATE $E' := E'\cup \{e'\}$
            \ENDIF
        \ENDFOR
    \ENDIF
\ENDFOR
\IF{$D'=(\mathcal{X}'\cup \{r\}, E')$ does not contain an $r$-arborescence $A'$}
\STATE go to Line 25.
\ELSE 
\STATE $\tilde{A}=\{e: e'\in A'\}$
\STATE $R=\{v\in V_G: |X_v| \geq 2, \delta^-(v)\cap\tilde{A}\neq \emptyset\}$
\FOR{each $v\in R$}
\STATE let $A_v$ be an $v$-arborescence in $(X_v,S(X_v))$
\ENDFOR
\RETURN $A^*=\tilde{A}\cup \bigcup_{v\in R} A_v$
\ENDIF
\RETURN{``No popular arborescence in $D$.''}
\end{algorithmic}
\end{algorithm}

A major difference from the algorithm without the vertex-weights \cite{kavitha20} appears in Line 10. 
If the condition shown in Line 10 is satisfied,  
there exists no popular arborescence  (see Lemma \ref{lemma28}).

The running time of the algorithm is $O(mn^2)$, which is the same as the algorithm without the vertex-weights \cite{kavitha20}.
Lines 1--6 take $O(mn^2)$ time and it is the bottleneck.
Lines 10--11 and 12--16 can be executed in $O(m)$ time, since 
they can be implemented to search each edge
at most once.

We now prove the validity of the algorithm described above by showing that 
\begin{itemize}
    \item if the algorithm returns an edge set $A^*$, then 
    $A^*$ is a popular arborescence in $D$ (Theorem \ref{theorem23}), and
    \item if $D$ admits a popular arborescence, then the algorithm returns an edge set $A^*$ (Theorem \ref{theorem19}).
\end{itemize}

The following lemma in \cite{kavitha20}
 is useful in our proof as well.
\begin{lemma}[\cite{kavitha20}]
\label{lemma7}
For each $v\in V_G$, let $X_v$ be the set defined in Lines 1--6 in Algorithm 1. 
Then, $\mathcal{X}$ is laminar, 
and $u\in X_v$ implies $X_u \subseteq X_v$.
\end{lemma}
 
\begin{theorem}
\label{theorem23}
If the algorithm returns an edge set $A^*$, then $A^*$ is a popular arborescence. 
\end{theorem}
\begin{proof}
First, we can show that $A^*$ is an $r$-arborescence in $D$ in the same manner as \cite{kavitha20}.

Next, we show that $A^*$ is a popular arborescence. 
For $X\in \mathcal{X}'$ such that $|X|\geq2$, let $v_{X} \in R$ be the terminal vertex of the edge in $A^* \cap \delta^-(X)$. 
Let $Y_{v_{X}}\subseteq X$ be a strongly connected component of the subgraph induced by 
$S(X)\cup \{e\in \delta^-(v_{X}) \colon e\succ_{v_{X}} A^*(v_{X})\}$
that contains $v_{X}$. 
For a vertex $t\in V_G$ such that $t\in X\setminus \{v_{X}\}$ for some $X\in \mathcal{X}$ with $|X|\geq 2$ or $\{t\}\in \mathcal{X}'$, let $Y_t = \{t\}$. 
Here, based on Proposition \ref{prop18}, we define
\begin{align}
y(Y)=
\begin{cases}
\label{y_prop7}
w(v) &(Y=Y_v \ \mathrm{for\ some}\  v \in V_G)\\
0 &(\mathrm{otherwise}).
\end{cases}
\end{align}
It is clear that $\sum_{Y\subseteq V_G}y(Y)=w(V_G)$. 
By Proposition \ref{lemma14}, the proof completes by showing that $y$ is a feasible solution to (LP2) determined by $A^*$.

We show that $y$ satisfies the constriant 
\eqref{lp2st1new}
of (LP2) for all edges.
First, we consider the edges in $\delta^-(v_{X})$ for each $X\in \mathcal{X}'$ with $|X|\geq 2$. The edges in $\delta^-(v_{X})\cap E[Y_{v_{X}}]$ do not enter any set in $\mathcal{F}(y)$.
For $e'\in \delta^-(v_X)\cap \delta^-(X)$, since $A^*(v_X)$ is not dominated by $e'$ from Lines 14--15 in Algorithm 1, it follows that $c_A(e')\in \{w(v_X), 2w(v_X)\}$ holds.
Since $Y_{v_{X}}$ is the only set in $\mathcal{F}(y)$ that $e'$ enters and  $y(Y_{v_{X}})=w(v_{X})$,
$y$ satisfies the constraint (\ref{lp2st1new}) for $e'$ in (LP2).
Consider an edge $f' \in \delta^-(v_{X})\cap \delta^-(Y_{v_{X}})\cap E[X]$.
By construction of $Y_{v_{X}}$, it must hold that $A^*(v_{X})\succsim_{v_X} f'$, and hence $c_A(f')\in \{w(v_{X}), 2w(v_{X})\}$. 
Since  $Y_{v_{X}}$ is the only set in $\mathcal{F}(y)$ that $f'$ enters,
it follows that $y$ satisfies the constraint (\ref{lp2st1new}) for $f'$ in (LP2).

Next, for $t\in Y_{v_X}\setminus \{v_X\}$, consider the edges in $\delta^-(t)$.
By our algorithm, $A^*(t)\in S(X)$. 
By the definition of safe edges, for $g\in \delta^-(t)\cap E[Y_{v_{X}}]$, it holds that $g\precsim_t A^*(t)$, and hence $c_A(g)=\{w(t), 2w(t)\}$. 
Since  $Y_{t}$ is the only set in $\mathcal{F}(y)$ that $g$ enters
and $y(Y_t)=w(t)$, $y$ satisfies the constraint (\ref{lp2st1new}) for $g$ in (LP2).

Then, consider an edge $g' \in \delta^-(t) \cap \delta^-(Y_{v_{X}})$.
Let $g'=(t_0,t)$. 
 If $t_0 \notin X$, then $A^*(t) \succ_t g'$ holds by the definition of $S(X)$. 
 If $t_0 \in X$, then $g'\notin S(X)$ holds by the definition of $Y_{v_{X}}$. 
 Since 
 $A^*(t) \in S(X)$ and 
 the preferences are given by a total preorder, 
 it follows that 
 \begin{align}
 \label{eq14}
 A^*(t)\succ_t g'
 \end{align}
 for any $g'\in \delta^-(t) \cap \delta^-(Y_{v_{X}})$, and thus $c_{A^*}(g')=2w(t)$.
An illustration is shown in Figure \ref{tproof}.

If $t\in \bar{X}$, then $w(v_{X}) \leq w(t)$ follows from the fact that $v_X$ is minimum weight in $\bar{X}$ by Line 13 of the algorithm.
Suppose that $t\in Y_{v_{X}} \setminus \bar{X}$. 
Since $t\notin \bar{X}$, $v_{X}$ is unreachable by safe edges from $t$. 
Furthermore, by construction of $Y_{v_X}$, there exists a path $P$ from $t$ to $v_X$ 
consisting of safe edges and the edges in $\delta^-(v_X)$ preferred to $A^*(v_X)$. 
In this path $P$, let $(t',v_X)$ be the edge in $\delta^-(v_X)$.
From our algorithm, $v_{X}$ is a vertex which does not satisfy at least one condition 1, 2, or 3 in Line 10, and $v_X$ satisfies the condition 2 from $(t',s)\succ_{v_X} A^*(v_{X})$.
Also, since the path $P\subseteq S(X)\cap (t',s)$, the condition 3 holds for $v_X$ and $(t',v_X)$.
Thus, it follows that $v_{X}$ does not satisfy the condition 1, and hence $w(v_{X})\leq w(t)$. From the above, $c_{A^*}(g')=2w(t) \geq w(v_{X})+w(t)=y(Y_{v_{X}})+y(Y_t)$ holds. Thus, $y$ satisfies the constraint (\ref{lp2st1new}) for $g'$ in (LP2).

Lastly, for $u\in X \setminus Y_{v_{X}}$, consider the edges in $\delta^-(u)$. By our algorithm, $A^*(u)\in S(X)$.
For an arbitrary edge $h\in \delta^-(u)$, $A^*(u)$ is not dominated by $h$ by the definition of safe edges, and therefore $c_{A^*}(h)\in \{w(u), 2w(u)\}$. 
Since  $Y_{u}$ is the only set in $\mathcal{F}(y)$ that $h$ enters
and $y(Y_u)=w(u)$, it follows that $y$ satisfies the constraint (\ref{lp2st1new}) for $h$ in (LP2).
 
Therefore we have proved that $y$ defined by (\ref{y_prop7}) satisfies the constraints in (LP2) for all edges.
\end{proof}

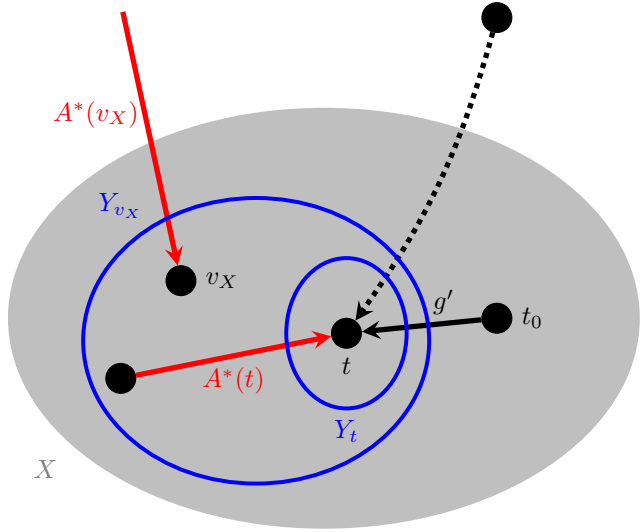
\begin{figure}
\begin{center}
\begin{tikzpicture}
\fill [lightgray] (5.7,3) circle [x radius=4.2, y radius=2.8];

\node (another1) [fill, draw, circle, inner sep=4] at (3,2.2){};
\node (another2) [fill, draw, circle, inner sep=4] at (8,7){};
\node (another3) at (3,7.2){};
\node (t0) [fill, draw, circle, inner sep=4, label=right:$t_0$] at (8,3){};
\node (t) [fill, draw, circle, inner sep=4, label=below:$t$] at (6,2.8){};
\node (vX) [fill, draw, circle, inner sep=4, label=right:$v_X$] at (3.8,3.5){};

\path[line width=2, ->,>=stealth] 
(another1) edge[red, below] node{$A^*(t)$} (t)
(another3) edge[red,above left] node{$A^*(v_X)$} (vX)
(t0) edge[above right] node{$g'$} (t);

\path[line width=2, ->,>=stealth, dotted] 
(another2) edge[left, bend left=10] node{} (t);

\draw [line width=1.5, blue] (t) circle [x radius=0.8, y radius=1];
\draw [line width=1.5, blue] (4.8,2.7) circle [x radius=2.3, y radius=1.9];

\node (Yt) [text=blue] at (6,1.5) {$Y_t$};
\node (Yv) [text=blue] at (3,4.5) {$Y_{v_X}$};
\node (X) [text=gray] at (2,1) {$X$};
\end{tikzpicture}
\end{center}
\caption{The edge $g'=(t_0, t)$ enters the two sets $Y_{v_{X}}$ and $Y_t$, and $A^*(t) \succ_t g'$ holds since $A^*(t)\in S(X)$.}
\label{tproof}
\end{figure}

We remark that Theorem \ref{theorem23} does not hold if the preferences are given by a partial order. 
In the case of partial orders, 
\eqref{eq14} cannot be derived 
from $A^*(t) \in S(X)$ and $g' \notin S(X)$, because it allows for incomparable edges.
This is the reason why we assume throughout the paper that the preferences are given by a total preorder.

Next, we prove that when a directed graph $D$ has a popular arborescence, the algorithm always finds one of them. Lemmas needed for the proof are given below.

Lemma \ref{lemma24} is shown similarly as Lemma 17 in \cite{kavitha20}, while our proof involves the vertex weights.
\begin{lemma}
\label{lemma24}
Let $A$ be a popular arborescence, 
let $y^*_{A}=\{Y_v \mid v\in V\}$ be the dual optimal solution determined by $A$ satisfying (\ref{prop17eq}), 
and let $\mathcal{X}' = \{X _v \colon v \in V_G\}$ be the family of sets defined in Lines 1--7 in Algorithm 1.
Then, $Y_v\subseteq X_v$ for each $v \in V_G$.
\end{lemma}

\begin{proof}
If $Y_v=\{v\}$, $Y_v\subseteq X_v$ is trivial. Consider the case where $|Y_v|\geq 2$. 

From Lemma \ref{lemma17} and Corollary \ref{cor21}, for $s\in Y_v \setminus \{v\}$, it follows that $Y_s = \{s\}$ and $y(\{s\})=w(s)$.
Furthermore, for the initial vertex $u$ of $A(s)$, we have that $c_A(u,s)=w(s)$, which implies $u\in Y_v$.

We assume to the contrary that $Y_v\setminus X_v\neq \emptyset$.
Let $i^*$ be the largest index satisfying $Y_v \subseteq X_v^{i^*}$ in Lines 3--5 of the algorithm for $v$.
Then, there exists a vertex in $Y_v\setminus X_v^{i^{*}+1}$, which cannot be reached from $v$ by the edges in $S(X_v^{i^*})$.
Meanwhile, any $v'\in Y_v\setminus \{v\}$ is reachable from $v$ by the edges in $A[Y_v]$.
Therefore, the vertex in $Y_v \setminus X_v^{i^{*}+1}$can be reached by an edge in $A$ from $v$, and $A$ must contain at least one edge in $(\delta^-(Y_v\setminus X_v^{i^{*} + 1})\cap \delta^+(X_v^{i^{*}+1}))\setminus S(X_v^{i^{*}}) $. 
Denote this edge by $(u,t)\in A$. 
Note that $Y_t = \{t\} \in \mathcal{F}(y^*_{A})$ follows from 
Corollary \ref{cor21}.
By construction of the sets $X_v^{i^*}$ and $X_v^{i^{*}+1}$, either one of the following two holds.
\begin{enumerate}
\item There exists an edge $(x_1,t)\in E[X_v^{i^*}]$ that dominates the edge $(u,t)$.
In this case, 
we have that $c_A(u, t)=w(t)$ and hence $c_A(x_1,t)=0$. 
Since we have 
$\{t\} \in \mathcal{F}(y^*_{A})$, this
violates the constraint (\ref{lp2st1new}) in (LP2).

\item There exists an edge $(x_2,t)\in \delta^-(X_v^{i^*})$ that is not dominated by the edge $(u,t)$.
Then, $c_A(x_2,t)\in \{0, w(t)\}$.
It follows from  $(x_2,t)\in \delta^-(X_v^{i^*})$ that
$(x_2,t)$ enter two sets $Y_v, Y_t\in \mathcal{F}(y^*_A)$. Now, since $y(Y_v)=w(v)$ and $y(Y_t)=w(t)$, it follows that $y(Y_v)+y(Y_t) > c_A(x_2,t)$. This 
violates
the constraint (\ref{lp2st1new}) in (LP2).
\end{enumerate}
From the above, 
we 
have shown
that $Y_v\subseteq X_v$ for any $v$.
\end{proof}

\begin{lemma}
\label{lemma25}
Let A be a populer arborescence in $D$ and let $X\in \mathcal{X}'$. Then, $A\cap \delta^-(X)$ contains only one edge, and $X=X_v$ holds for the terminal vertex $v$ of that edge. 
\end{lemma}
\begin{proof}
Clearly, $A\cap \delta^-(X)\neq \emptyset$. Let $(u,v)\in A\cap \delta^-(X)$. For each $X\in \mathcal{X}'$, let $\bar{X}$ be the strongly connected component of $(X, S(X))$ such that an edge in $S(X)$ does not enter. 
Note that $X_s=X$ for each $s\in \bar{X}$.
Let $y^*_{A}$ be the dual optimal solution determined by $A$ satisfying (\ref{prop17eq}), and 
let $\mathcal{F}(y^*_{A})=\{Y_v\mid v\in V\}$.

First, we will show that $\bar{X} \subseteq Y_v$.
We assume to the contrary that there 
exists a vertex $s\in \bar{X}\setminus Y_v$. Since $\bar{X}$ is strongly connected, there exists a path $P\subseteq S(X)$ from $s$ to each vertex in $X$ by the edges in $S(X)$. 
Let $e$ be an edge in $P$ that enters $Y_v$.
If the terminal vertex of $e$ is $v$, since $e\in S(X)$ and $(u,v)\in A\cap \delta^-(X)$, 
it must hold that $e\succ_v (u,v)$, which implies
$0=c_A(e)<c_A(u,v)=w(v)$.
However, it also holds that  $c_A(e)\ge y^*_{A}(Y_v)=w(v)$, a contradiction.
If the terminal vertex of $e$ is not $v$, let $t$ be the terminal vertex, that is, $e\in \delta^-(Y_v)\cap \delta^-(\{t\})$. Since $e\in S(X)$, $e$ is not dominated by any adge in $\delta^-(t)$, and hence, $c_A(e)\leq w(t)$.
Recall that $y^*_{A}(Y_v)=w(v)$ and $y(Y_t)=w(t)$ (Proposition \ref{prop18}). 
Therefore we have $c_A(e) < w(t) +w(v) =y^*_{A}(Y_t)+y^*_{A}(Y_v)$, 
violating  the constraint \eqref{lp2st1new} in (LP2). 
We thus conclude that $\bar{X}\subseteq Y_v$.

Next, we show that $v\in \bar{X}$.
By the definition of $\bar{X}$, the edges in $S(X)$ does not enter $\bar{X}$. Thus, for vertex $t$ such that $X=X_t$, it holds that $t\in \bar{X}$.
We show that $v\in \bar{X}$, by deriving $X=X_v$. 
For each $s\in \bar{X}$, from $\bar{X}\subseteq Y_v$ shown above and Lemma \ref{lemma24}, we obtain $s\in X_v$. Thus, by 
Lemma \ref{lemma7}, we have $X=X_s\subseteq X_v$.
Since $X\in \mathcal{X}'$ is maximal in $\mathcal{X}$, it holds that $X=X_v$ and thus, $v\in \bar{X}$.
 
Finally, we show that $|A\cap\delta^-(X)|=1$. 
Assume to the contrary that there exist two different edges $(u,v), (u',v')\in A \cap \delta^-(X)$. 
Note that $v\neq v'$.
Now, from the above argument, it follows that $\bar{X}\subseteq Y_v \cap Y_{v'}$. Thus, by the laminarity of  $\mathcal{F}(y_A^*)$, 
w.l.o.g.\ we assume 
$Y_v\subseteq Y_{v'}$.
Then, from $(u,v)\in A \cap\delta^-(X)$, and $Y_{v}\subseteq Y_{v'}\subseteq X$, we derive $(u,v)\in \delta^-(Y_v)\cap \delta^-(Y_{v'})$. 
However,  
this violates 
the constraint \eqref{lp2st1new} of (LP2), 
because 
$c_A(u,v)=w(v)$, and $y^*_{A}(Y_v)=w(v)$ and $y^*_{A}(Y_{v'})=w(v')$ by  Proposition \ref{prop18}.
Therefore we conclude that $|A\cap\delta^-(X)|=1$.
\end{proof}

\begin{lemma}
\label{lemma27}
Let $A$ be an $r$-arborescence in $D$ and let $X\in \mathcal{X}'$. If there exists an edge in 
$A[X]\setminus S(X)$,
then there is an $r$-arborescence more popular than $A$. 
\end{lemma}
\begin{proof}
By Lemma \ref{lemma25}, if $|A\cap \delta^-(X)|\neq 1$, then there is an $r$-arborescence more popular than $A$.
We thus assume that $|A\cap \delta^-(X)|=1$. Let $v\in X$ be the terminal vertex of the edge in $A\cap \delta^-(X)$.
Again by Lemma
\ref{lemma25}, $X\neq X_v$ implies that there is an $r$-arborescence more popular than $A$, 
and hence we further assume that $X=X_v$.

Denote the edge in $A[X] \setminus S(X)$ by $f=(s,t)$.
Then, it follows from $A(v)\in \delta^-(X)$ that $t\neq v$. 
Since $X=X_v$, we can construct a $v$-arborescence in a subgraph $(X,S(X))$.
Define an $r$-arborescence $A'$ in $D$ by replacing 
the edges in $A$ entering each vertex in $X \setminus \{v\}$
with this $v$-arborescence.
Now, since $A'(t)\in S(X)$ and $f \notin S(X)$, we have $A'(t)\succ_t f$.
Also, observe that 
$A'(t')\succsim_{t'} A(t')$
if $t'\in X \setminus \{t\}$ and $A'(t')=A(t')$ if $t\in V\setminus X$.
Thus, $\Delta_w(A',A)\geq w(t) > 0$, 
implying that $A'$ is more popular than $A$.
\end{proof}

\begin{lemma}
\label{lemma29}
For a popular arborescence $A$ and $X\in \mathcal{X}'$, let $\bar{X}$ be the strongly connected component 
in the subgraph $(X, S(X))$
that an edge in $S(X)$ does not enter. Then the terminal vertex $v$ of the edge in $A\cap\delta^-(X)$ 
belongs to $\bar{X}$ and
has minimum weight in $\bar{X}$.
\end{lemma}
\begin{proof}
It follows from Lemma \ref{lemma25} that $A\cap \delta^-(X)$ contains only one edge, and denote it by $(t,v)$. 
Again by Lemma \ref{lemma25}, 
we have $X = X_v$, and therefore, $v \in \bar{X}$.
Let $y_A^*$ be the dual solution satisfying \eqref{prop17eq}.
First, in the proof of Lemma \ref{lemma25}, we have shown that $\bar{X}\subseteq Y_v$.
Next, we show that $v$ has the minimum weight in $Y_v$. 
Let $u\in Y_v\setminus \{v\}$.
Then, from Corollary \ref{cor21}, we have $Y_u=\{u\}$.
By Proposition \ref{prop18}, we have that $y^*_{A}(Y_v)=w(v)$ and $y^*_{A}(Y_u)=w(u)$. 
Since $(r,u)$ enters $Y_v$ and $Y_u$, it follows that
\begin{align*}
&c_A(r,u)\geq y^*_{A}(Y_v)+y^*_{A}(Y_u)\\
\Leftrightarrow \  &2w(u)\geq w(v)+w(u)\\
\Leftrightarrow \  &w(u)\geq w(v).
\end{align*}
Thus, it is shown that $v$ has minimum weight in $Y_v$ and therefore it follows from $\bar{X}\subseteq Y_v$ that $v$ has the minimum weight in $\bar{X}$.
\end{proof}

\begin{lemma}
\label{lemma28}
For some $v\in V_G$, if all the vertices of minimum weight in $\bar{X}_v$ satisfy the condition of Line 10 in Algorithm 1, then there is no popular arborescence.
\end{lemma}
\begin{proof}
Let $v\in V_G$ and $\bar{M}_{v}$ be the set of the vertices with the minimum weight in $\bar{X}_v$.
Assume to the contrary that 
every vertex $v'\in \bar{M}_v$ satisfies the condition of Line 10, 
i.e.,\ there exist $s_{v'}\in X_v\setminus \bar{X}_v$ and $f_{v'} \in (E[X_v]\setminus S(X_v))\cap \delta^-(v')$ satisfying the three conditions, and there exists a popular arborescence $A$.

First, consider the case when $|\bar{X}_v|=1$
i.e., $v=v'$.
By the definition of $\bar{X}_v$, there is no safe edge that enters $v'$ since all vertices in $X_v$ are reachable by safe edges from $v'$.
In this case, however, it follows from condition 3 that the most preferred edge in $\delta^-(v) \cap E[X_v]$ satisfies the definition of safe edges, 
a contradiction.

Next, suppose that $|\bar{X}_v|\geq 2$. 
Note that $A[X_v]\subseteq S(X_v)$ by Lemma \ref{lemma27}. 
Let $u$ be the entry-point of $X_v$ for $A$. 
By Lemma \ref{lemma29}, we have $u\in \bar{M}_v$ and hence it follows that there exist $s_u\in X_v \setminus \bar{X}_v$ and $f_u \in (E[X_v]\setminus S(X_v))\cap \delta^-(u)$ satisfying 
the three conditions
in Line 10.
Let $B$ be the $r$-arborescence obtained from $A$ by replacing the edges in 
$A[X_v] \cup A(u)$
with $(r, s_u)$ and $s_u$-arborescence in $(X_v, S(X_v) \cup \{f_u\})$.
It then follows from $\Delta_w (B,A)=w(u)-w(s_u)>0$ that $B$ is more popular than $A$, which contradicts that $A$ is a popular arborescence.
\end{proof}

\begin{theorem}
\label{theorem19}
If $D$ admits a popular arborescence, then our algorithm finds one.
\end{theorem}
\begin{proof}
Let $A$ be a popular arborescence and let $y_A^*$ be the dual optimal solution satisfying (\ref{prop17eq}). 
By Lemma \ref{lemma25}, for each $X\in \mathcal{X}'$, it holds that $|A\cap \delta^-(X)|=1$.
Let $e_X=(u,v)\in A\cap\delta^-(X)$.
Then $X=X_v$ holds by Lemma \ref{lemma25} and thus, we have $v\in \bar{X}$.

First, we show that $e_X=(u,v)$ 
is not dominated by any adge $(u',v)\in \delta^-(X)$.
By Lemma \ref{lemma24}, we have $Y_v\subseteq X_v =X$. If $(u',v)\in \delta^-(X)$ dominates $(u,v)$, then $c_A(u',v)=0$. 
However, since $(u',v)\in \delta^-(Y_v)$ holds, this 
violates the constraint \eqref{lp2st1new} in (LP2).
Thus, $e_X=(u,v)$ is not dominated by an edge in $\delta^-(X)\cap \delta^-(v)$.

From Lemma \ref{lemma29}, the vertex $v$ has the minimum weight in $\bar{X}$. 
Furthermore, from Lemma \ref{lemma28}, when $D$ admits a popular arborescence, there exists a vertex that has minimum weight in $\bar{X}$ which does not satisfy conditions of Line 10 in Algorithm 1.
For such a vertex $v'$, our algorithm adds to $E'$ an edge in $\delta^-(X) \cap \delta^-(v')$ that is not dominated by any edge in $\delta^-(X)$.
Thus, we have $e_X \in E'$.
Since there exists an $r$-arborescence in $D$ and $e_X \in E'$ holds, 
for the graph $D'$ constructed in Line 17 of the algorithm, each $X \in \mathcal{X}'$ is reachable from $r$ using the edges in $E'$.
Hence $E'$ contains an $r$-arborescence $A'$ in $D'$.
Therefore, the algorithm returns an edge set $A^*$, which is a popular arborescence in $D$ by Theorem \ref{theorem23}.
\end{proof}

\section{Conclusion}
In this paper, we have provided an algorithm for finding a weighted popular branching,
which extends the algorithm of Kavitha et al., when the weights of each vertex satisfy condition \eqref{3.11} and the preferences of each vertex are given by a total preorder. 

An apparent future work is to analyze the computational complexity of the weighted popular branching problem
in which the vertex weights do not necessarily satisfy condition \eqref{3.11} and the preferences are given by a partial order.
It is also of interest to extend the arguments for \emph{unpopularity margin}, 
\emph{unpopularity factor}, 
and \emph{popular branching polytope} 
in \cite{kavitha20} to weighted popular branchings.


\begin{thebibliography}{99}
\bibitem{abraham07}Abraham, D.J., Irving, R.W., Kavitha, T., Mehlhorn, K.: Popular matchings. SIAM Journal on Computing 37, 1030--1045, 2007.

\bibitem{BP10}Bir\'{o}, P., Irving, R.W., Manlove, D.F.: Popular matchings in the marriage and
roommates problems.
Proceedings of the 7th International Conference on Algorithms
and Complexity 
(CIAC 2010), Lecture Notes in Computer Science 6078, 97--108,
Springer, 2010.


\bibitem{Cseh17}Cseh, \'{A}.,
Huang, C.-C., 
Kavitha, T.: Popular matchings with two-sided preferences
and one-sided ties. SIAM Journal on Discrete Mathematics 31(4), 2348--2377, 2017.

\bibitem{dominantM17}Cseh, \'{A}., Kavitha, T.: Popular edges and dominant matchings. Mathematical Programming
172(1), 209--229, 2017.

\bibitem{EandG77} Edmonds, J., Giles, R.: A min-max relation for submodular functions on graphs. Studies in Integer Programming, Annals of Discrete Mathematics, 185--204, 1977.

\bibitem{frank79b}Frank, A.: Kernel systems of directed graphs. 
Acta 
Scientiarum
Mathematicarum 41, 63--76, 1979.

\bibitem{fulkerson74}Fulkerson, D. R.: Packing rooted directed cuts in a weighted directed graph. Mathematical Programming 6(1), 1--13, 1974.

\bibitem{quasiPM19}Fenza, Y., Kavitha, T.: Quasi-popular matchings, optimality, and extended formulations.
Proceedings of the 31st Annual ACM-SIAM Symposium on Discrete Algorithms 
(SODA 2020), 325--344, 2020.

\bibitem{gardenfors75}
G\"{a}rdenfors, P.: Match making: Assignments based on bilateral preferences. 
Behavioral Sciences 20(3), 166--173,
1975.

\bibitem{OPM09}Kavitha, T., Nasre, M.: Optimal popular matchings. 
Discrete Applied Mathematics, 
157(14), 
3181--3186, 2009.

\bibitem{kavitha20} Kavitha, T., Kir\'{a}ly, T., Matuschke, J., Schlotter, I., Schmidt-Kraepelin, U.: Popular branchings and their dual certificates.
Mathematical Programming, to appear.

\bibitem{wpm06}
Mestre, J.: Weighted popular matchings. ACM Transactions on Algorithms  10(1), 2:1-2:16, 2014.
\end{thebibliography}
\end{document}